\documentclass[a4paper]{amsart}
\usepackage{amsmath,amssymb,amsthm,geometry}
\usepackage[all]{xy}

\DeclareMathOperator{\im}{\mathsf{Im}}

\DeclareMathOperator{\Hom}{\mathsf{Hom}}
\DeclareMathOperator{\sHom}{\underline{\Hom}}
\DeclareMathOperator{\End}{\mathsf{End}}
\DeclareMathOperator{\Ext}{\mathsf{Ext}}
\DeclareMathOperator{\id}{id}
\DeclareMathOperator{\gd}{gl.dim}
\DeclareMathOperator{\dd}{dom.dim}
\DeclareMathOperator{\md}{\mathsf{mod}}
\DeclareMathOperator{\smd}{\underline{\md}}
\DeclareMathOperator{\proj}{\mathsf{proj}}
\DeclareMathOperator{\add}{\mathsf{add}}
\DeclareMathOperator{\CM}{\mathsf{CM}}
\DeclareMathOperator{\sCM}{\underline{\CM}}

\def\der{\mathsf{D^{\mathrm{b}}}}
\def\K{\mathsf{K^{\mathrm{b}}}}
\def\rHom{\mathrm{Hom}}

\renewcommand\Im{\im}

\def\A{\mathcal{A}}
\def\C{\mathcal{C}}

\def\P{\mathcal{P}}
\def\T{\mathcal{T}}

\def\Z{\mathbb{Z}}

\def\a{\alpha}
\def\s{\sigma}
\def\Lam{\Lambda}
\def\Om{\Omega}
\def\op{\mathrm{op}}
\def\Gam{\Gamma}

\def\cK{\mathcal{K}}

\title{Auslander Correspondence for Triangulated Categories}
\author{Norihiro Hanihara}
\subjclass[2010]{18E30, 16E65, 16G70, 16E05}
\keywords{triangulated category; Auslander correspondence; periodic algebra; Cohen-Macaulay module}
\address{Graduate School of Mathematics, Nagoya University, Chikusa-ku, Nagoya, 464-8602, Japan}
\email{m17034e@math.nagoya-u.ac.jp}

\newtheorem{Thm}{Theorem}[section]
\newtheorem{Lem}[Thm]{Lemma}
\newtheorem{Prop}[Thm]{Proposition}

\newtheorem{Cor}[Thm]{Corollary}

\theoremstyle{definition}
\newtheorem{Def}[Thm]{Definition}
\newtheorem{Ex}[Thm]{Example}

\theoremstyle{remark}
\newtheorem{Rem}[Thm]{Remark}
\renewcommand{\theprf}

\begin{document}
\begin{abstract}
We give analogues of the Auslander correspondence for two classes of triangulated categories satisfying certain finiteness conditions. The first class is triangulated categories with additive generators and we consider their endomorphism algebras as the Auslander algebras. For the second one, we introduce the notion of $[1]$-additive generators and consider their graded endormorphism algebras as the Auslander algebras. We give a homological characterization of the Auslander algebras for each class. Along the way, we also show that the algebraic triangle structures on the homotopy categories are unique up to equivalence.
\end{abstract}

\maketitle

\section{Introduction}\label{intro}
The main concern in representation theory of algebras is to understand the module categories. Among such categories, those with {\it{finitely many}} indecomposable objects, or equivalently the {\it{representation-finite}} algebras, are most fundamental. Let us recall the following famous theorem due to Auslander:
\begin{Thm}[\cite{A}, Auslander correspondence]\label{Au}
There exists a bijection between the set of Morita equivalence classes of finite dimensional algebras $\Lam$ of finite representation type and the set of Morita equivalence classes of finite dimensional algebras $\Gamma$ such that $\gd \Gamma \leq 2$ and $\dd \Gamma \geq 2$.
\end{Thm}

This theorem states that a categorical property (=representation-finiteness) of $\md \Lam$ can be characterized by homological invariants (=$\gd$ and $\dd$) of $\Gamma$, called the {\it{Auslander algebra}} of $\md\Lam$. Such relationships between categorical properties of those appearing naturally in representation theory, and homological properties of their `Auslander algebras' has been established in \cite{Iy1,Iy2,E}.

The aim of this paper is to find an analogue of these results for triangulated categories \cite{Nee}. Let $k$ be an arbitrary field and $\T$ be a $k$-linear, $\rHom$-finite, idempotent-complete triangulated category. We consider two kinds of finiteness conditions on triangulated categories. 

The first one is a direct analogue of representation-finiteness:
$\T$ is {\it{finite}}, that is, $\T$ has finitely many indecomposable objects up to isomorphism.
In this case, $\T$ has an additive generator $M$. We call $\End_\T(M)$ the {\it{Auslander algebra}} of $\T$, which is uniquely determined by $\T$ up to Morita equivalence. The first main result of this paper is the following homological characterization of the Auslander algebras of triangulated categories. We say that a finite dimensional algebra $A$ is {\it twisted $n$-periodic} if it is self-injective and there exists an automorphism $\a$ of $A$ such that $\Om^n \simeq (-)_\a$ as functors on $\smd A$. We refer to Corollary \ref{twgss} for equivalent characterizations.
\begin{Thm}\label{intf}
Let $k$ be a perfect field. The following are equivalent for a basic finite dimensional $k$-algebra $A$:
\begin{enumerate}
\item\label{triang} $A$ is the Auslander algebra of a $k$-linear, $\rHom$-finite, idempotent-complete triangulated category which is finite.
\item\label{tw3} $A$ is twisted $3$-periodic.
\end{enumerate}
\end{Thm}

This result shows a close connection between periodic algebras \cite{ES} and triangulated categories.
Our proof depends on Amiot's result (Proposition \ref{Ami}). This is a complement of Heller's classical observation \cite[16.4]{He} which gives a parametrization of pre-triangle structures on a pre-triangulated category $\T$ in terms of isomorphisms $\Om^3 \simeq [-1]$ on $\smd\T$. Later practice of this property of the third syzygy in representation theory can be seen in \cite{AR,Y2,Am,IO}.

Moreover, with some additional assumptions on $\T$, we give a bijection between finite triangulated categories and certain algebras, which is a more precise form of the above theorem; see Theorem \ref{std}.

The second finiteness condition is the following:
\begin{enumerate}
\def\theenumi{(S{\arabic{enumi}})}
\def\labelenumi{\theenumi}
\item\label{shift} There is an object $M \in \T$ such that $\T=\add\{M[n] \mid n \in \Z \}$.
\item\label{simconn} For any $X,Y \in \T$, $\Hom_\T(X,Y[n])=0$ holds for almost all $n$.
\end{enumerate}

If these conditions are satisfied, we say $\T$ is {\it{$[1]$-finite}} and call $M$ as in \ref{shift} a {\it{$[1]$-additive generator}}. For example, the bounded derived categories of representation-finite hereditary algebras are $[1]$-finite, and additive generators for module categories are $[1]$-additive generators for the derived categories. There are various studies on $[1]$-finite triangulated categories, for example \cite{Ro, XZ, Am}. Note that $[1]$-finite triangulated categories have infinitely many indecomposable objects unless $\T=0$.


For a $[1]$-finite triangulated category $\T$ with a $[1]$-additive generator $M$, we call
\[ C=\bigoplus_{n \in \Z}\Hom_\T(M,M[n]) \]
the {\it{$[1]$-Auslander algebra}} of $\T$, which is naturally a $\Z$-graded algebra and is uniquely determined by $\T$ up to graded Morita equivalence. Thanks to our condition \ref{simconn}, $C$ is finite dimensional. To study it, we prepare some results on `graded projectivization' in Section \ref{easy} (see Proposition \ref{efu}). Such constructions of graded algebras appear naturally in various contexts \cite{AZ,As}.

Our second main result is the Auslander correspondence for $[1]$-finite triangulated categories. To state it, we have to restrict to a nice class of triangulated categories called {\it algebraic}. Recall that they are the stable categories of Frobenius categories \cite[I.2.6]{Hap}. Algebraic triangulated categories are enhanced by differential graded categories \cite{Ke2}, and play a central role in tilting theory \cite{AHK}.

Now we can formulate the following second main result of this paper in terms of algebraic triangulated categories and graded algebras.
We say that a finite dimensional $\Z$-graded algebra $A$ is {\it $(a)$-twisted $n$-periodic} if it is self-injective and there exists a graded automoprhism $\a$ of $A$ such that $P_\a \simeq P$ for all $P \in \proj^\Z\!A$ and $\Om^n \simeq (-)_\a(a)$ as functors on $\smd^\Z\!A$. We refer to Corollary \ref{grgss} for equivalent conditions.
\begin{Thm}\label{int1}
Let $k$ be an algebraically closed field. There exists a bijection between the following.
\begin{enumerate}
\item\label{T} The set of triangle equivalence classes of $k$-linear, $\rHom$-finite, idempotent-complete, algebraic triangulated categories $\T$ which are $[1]$-finite.
\item\label{A} The graded Morita equivalence classes {\rm (}see Definition {\rm \ref{grmo})} of finite dimensional graded $k$-algebra $C$ which are $(-1)$-twisted $3$-periodic.
\item\label{Dyn} A disjoint union of Dynkin diagrams of type A, D, and E.
\end{enumerate}
The correspondences are given as follows:
\begin{itemize}
\item From {\rm ({\ref{T}})} to {\rm ({\ref{A}})}: Taking the $[1]$-Auslander algebra of $\T$.
\item From {\rm (\ref{T})} to {\rm (\ref{Dyn})}: Taking the tree type of the AR-quiver of $\T$.
\item From {\rm ({\ref{A}})} to {\rm ({\ref{T}})}: $C \mapsto \proj^\Z\!C$.
\item From {\rm (\ref{Dyn})} to {\rm (\ref{T})}: $Q \mapsto k(\Z Q)$, where $k(\Z Q)$ is the mesh category associated with $\Z Q$.
\end{itemize}
\end{Thm}
Moreover, we have the following explicit descriptions of (\ref{T}) and (\ref{A}) in the above theorem.
\begin{Thm}[Theorem \ref{cor}, Proposition \ref{easy one}]\label{desc}
The classes {\rm (\ref{T})} and {\rm (\ref{A})} in Theorem \ref{int1} are the same as {\rm (\ref{T}$^\prime$)} and {\rm (\ref{A}$^\prime$)}, respectively. 
\begin{enumerate}
\renewcommand{\labelenumi}{(\arabic{enumi}$^\prime$)}
\item The set of triangle equivalence classes of the bounded derived categories $\der(\md kQ)$ of the path algebra $kQ$ for a disjoint union $Q$ of Dynkin quivers of type A, D, and E.
\item The orbit algebras $k(\Z Q)/[1]$ for a disjoint union $Q$ of Dynkin quivers of type A, D, and E.
\end{enumerate}
\end{Thm}

Comparing with Theorem \ref{intf}, Theorem \ref{int1} is more strict in the point that the Auslander algebras $C$ corresponds {\it{bijectively}} to the triangulated categories.
This can be done by the classification of $[1]$-finite triangulated categories as is stated in ($\ref{T}^\prime$). These results suggest that $[1]$-finite triangulated categories are easier than finite ones in controlling their triangle structures as well as their additive structures.

Our classification is deduced from the following uniqueness of the triangle structures on the homotopy categories.
\begin{Thm}[Theorem \ref{unique}]\label{kb}
Let $\Lam$ be a ring such that $\K(\proj\Lam)$ is a Krull-Schmidt category and $\Lam$ does not have a semisimple ring summand, and let $\C$ be an algebraic triangulated category. If $\C$ and $\K (\proj\Lam)$ are equivalent as additive categories, then they are equivalent as triangulated categories.
\end{Thm}
For example, $\K(\proj\Lam)$ is Krull-Schmidt if $\Lam$ is a module-finite algebra over a complete Noetherian local ring.
We actually see that the possible triangle structure on a given Krull-Schmidt additive category is unique in the sense that the suspensions and the mapping cones are uniquely determined as objects, see Proposition \ref{aaaa} for details.

As an application of our classification Theorem \ref{desc} of $[1]$-finite triangulated categories, we recover the main result of \cite{CYZ} stating that any finite dimensional algebra over an algebraically closed field with derived dimension 0 is piecewise hereditary of Dynkin type.

We also apply Theorem \ref{desc} to Cohen-Macaulay representation theory.
A rich source of $[1]$-finite triangulated categories is given by CM-finite Iwanaga-Gorenstein algebras \cite{CR1, CR2, LW, S, Y}, for example, simple singularities and trivial extension algebras of representation-finite hereditary algebras. We consequently obtain the following result, which states that $\sCM^\Z\!\Lam$ is triangle equivalent to the derived category of a Dynkin quiver under some mild assumptions.
\begin{Cor}[Theorem \ref{CM}]\label{appli}
Let $k$ be an algebraically closed field and $\Lam=\bigoplus_{n \geq 0}\Lam_n$ be a positively graded CM-finite Iwanaga-Gorenstein algebra such that each $\Lam_n$ is finite dimensional over $k$ and $\Lam_0$ has finite global dimension. Then, the stable category $\sCM^\Z\!\Lam$ is $[1]$-finite and therefore, it is triangle equivalent to $\der(\md kQ)$ for a disjoint union $Q$ of some Dynkin quivers of type A, D, and E.
\end{Cor}

This partially recovers \cite{KST}, \cite[2.1]{BIY} in a quite different way. Note that our result is more general, but less explicit in the sense that Corollary \ref{appli} does not give the type of $Q$ from given $\Lam$.

As this application suggests, our classification shows that the `easiest' triangulated categories are very likely to be the derived category of Dynkin quivers, and provides a completely different method (from a direct construction of tilting objects) of giving a triangle equivalence for such categories.

\subsection*{Acknowledgement}
The author is deeply grateful to his supervisor Osamu Iyama for many helpful suggestions and careful instructions.

\subsection*{Notations and conventions}\label{conv}
We denote by $k$ a field. 
For a category $\C$, we denote by $\Hom_\C(-,-)$ or simply $\C(-,-)$ the $\rHom$-spaces between the objects and by $J_\C(-,-)$ the Jacobson radical of $\C$. A $\C$-module is a contravariant functor from $\C$ to the category of abelian groups. A $\C$-module $M$ is finitely presented if there is an exact sequence $\C(-,X) \to \C(-,Y) \to M \to 0$ for some $X,Y \in \C$. We denote by $\md\C$ the category of finitely presented $\C$-modules. If $\C$ is graded by a group $G$, the category of finitely presented graded functor is denoted by $\md^G\!\C$, and its projectives by $\proj^G\!\C$. The morphism space in $\md^G\!\C$ is denoted by $\Hom_\C(-,-)_0$ or $\C(-,-)_0$. The category $\md^G\!\C$ is endowed with the grade shift functor $(g)$ for each $g \in G$, defined by $M(g)=M$ as an ungraded module and $(M(g)(X))_h=(MX)_{gh}$ for each $X \in \C$.

Similarly, for a $k$-algebra $A$, the Jacobson radical of $A$ is denoted by $J_A$. A module over $A$ means a finitely generated right module. We denote by $\md A$ (resp. $\proj A$) the category of (projective) $A$-modules. If $A$ is graded, the category of graded (projective) $A$-modules is denoted by $\md^G\!A$ (resp. $\proj^G\!A$).

\section{Periodicity of syzygies}\label{not easy}
Let $A$ be a $k$-algebra. We denote by $A^e$ the enveloping algebra $A^\op \otimes_k A$ and by $\Om_A$ (resp. $\Om_{A^e}$) the syzygy, that is, the kernel of the projective cover in $\md A$ (resp. $\md A^e$).
In this section, we generalize for our purpose the result of Green-Snashall-Solberg \cite{GSS} which relates the periodicity of syzygy of simple $A$-modules and that of $A$ considered as a bimodule over itself.
The following theorem and its proof is a graded and twisted version of \cite[1.4]{GSS}.
\begin{Thm}\label{GSS}
Let $G$ be an abelian group and $A$ be a finite dimensional, ring-indecomposable, non-semisimple $G$-graded $k$-algebra. Assume that $J_A=J_{A_0} \oplus (\bigoplus_{i \neq 0} A_i)$ and that $A/J_A$ is separable over $k$. Then, the following are equivalent for $a \in G$ and $n>0$.
\begin{enumerate}
\item\label{sim} $\Om^n_A(A/J_A) \simeq A/J_A(a)$ in $\md^G\!A$.
\item\label{twisted} $A$ is self-injective and there exists a graded algebra automorphism $\a$ of $A$ such that $\Om^n \simeq (-)_\a(a)$ as functors on $\smd^G\!A$.
\item\label{bi} There exists a graded algebra automorphism $\a$ of $A$ such that $\Om^n_{A^e}(A) \simeq {}_1A_\a(a)$ in $\md^G\!A^e$. 
\end{enumerate}
\end{Thm}
\begin{proof}
By the original case, we have that $A$ is self-injective under the assumption (\ref{bi}). Then the implication (\ref{bi}) $\Rightarrow$ (\ref{twisted}) follows. Also, (\ref{twisted}) $\Rightarrow$ (\ref{sim}) is clear.

It remains to prove ({\ref{sim}}) implies ({\ref{bi}}). 
Note that by our assumption on $J_A$, it is graded and any simple object in $\md^G\!A$ is simple in $\md A$.
Assume (\ref{sim}) holds and set $B=\Om^n_{A^e}(A)$. This is a projective $A$-module on each side.

{\it{Step 1: $S \otimes_A B$ is simple for all graded simple (right) $A$-modules $S$.}}

Let $S$ be a graded simple $A$-module. Then, applying $S \otimes_A -$ to the minimal projective resolution $P \colon \cdots \to P_i \xrightarrow{d_i} P_{i-1} \to \cdots \to P_0$ of $A$ in $\md^G\!A^e$ yields the minimal projective resolution of $S$ in $\md^G\!A$. Indeed, since $A/J_A$ is separable over $k$, we have $J_{A^e}=J_A \otimes_k A + A \otimes_k J_A$. Then, $\Im d_i \subset P_{i-1}J_{A^e}=J_AP_{i-1}+P_{i-1}J_A$ by the minimality of $P$ and therefore, $\Im(S \otimes d_i) \subset S \otimes_A P_{i-1}J_A$ by $S \otimes_A J_AP_{i-1}=0$. This shows $S \otimes_A P$ is minimal. Therefore we have $S \otimes_A B \simeq \Om_A^n(S)$, which is simple by assumption (\ref{sim}).
 
It follows by induction that the exact functor $-\otimes_A B$ preserves length.

{\it{Step 2: $B \simeq A(a)$ in $\md^G\!A$.}}

Consider the exact sequence $0 \to J_A \to A \to A/J_A \to 0$ in $\md^G\!A$. Applying $-\otimes_A B$ yields $B \to A/J_A(a) \to 0$. This shows that the module $B$ contains $A/J_A(a)$ in its top. But since $B$ is a projective (right) $A$-module having the same length as $A$ by the remark following Step 1, we see that $B \simeq A(a)$ in $\md^G\!A$.

%
%
%
%
%
{\it{Step 3: There exists a graded algebra automorphism $\a$ of $A$ such that $\Om^n_{A^e}(A) \simeq {}_1A_\a(a)$.}}

By Step 2, there exists a graded algebra endomorphism $\a$ of $A$ such that $B \simeq {}_\a A_1(a)$ in $\md^G\!A^e$. Indeed, fix an isomorphism $\varphi \colon A(a) \to B$ in $\md^G\!A$, put $x=\varphi(1)$, and set $\a(u)=\varphi^{-1}(ux)$ for $u \in A$. Then, $\a$ is of degree $0$, since $x$ and $\varphi$ are, and it is easily checked that $\a$ is an algebra endomorphism and that $\varphi \colon {}_\a A_1(a) \to B$ is an isomorphism in $\md^G\!A^e$.
Now we show that $\a$ is an isomorphism. Let $I$ be the kernel of $\a$. Since $B \simeq {}_\a A$ is a projective left $A$-module, the inclusion $I \subset A$ in $\md^G\!A$ stays injective by applying $-\otimes_A {}_\a A$. But since that map $I \otimes_A {}_\a A \to {}_\a A$ is zero, we have $I \otimes_A {}_\a A=0$, and we conclude that $I=0$ by the remark following Step 1.

This finishes the proof of ({\ref{sim}}) $\Rightarrow$ ({\ref{bi}}).
\end{proof}

We need the following two particular cases. 
The first one, which we will use in Section \ref{8} is the following result for $G=\{1\}$, where the special case `$\Om^n(S) \simeq S$ for all simples' is \cite[1.4]{GSS}. 
\begin{Cor}\label{twgss}
Let $A$ be a ring-indecomposable, non-semisimple finite dimensional $k$-algebra such that $A/J_A$ is separable over $k$. Then, the following are equivalent for $n>0$.
\begin{enumerate}
\item\label{s} $\Om^n_A(A/J_A) \simeq A/J_A$.
\item\label{func} $A$ is self-injective and there exists an automorphism $\a$ of $A$ such that $\Om^n \simeq (-)_\a$ as functors on $\smd A$.
\item\label{perio} There exists an automorphism $\a$ of $A$ such that $\Om_{A^e}^n(A) \simeq {}_1A_\a$ in $\md A^e$.
\end{enumerate}
\end{Cor}

We call such algebras as follows.
\begin{Def}
A finite dimensional algebra is {\it twisted $n$-periodic} if it is a direct product of simple algebras or algebras satisfying the equivalent conditions in Corollary \ref{twgss}.
\end{Def}

The second one is the following for $G=\Z$ and the permutation of simples is the identity, which will be used in Section \ref{main}.
\begin{Cor}\label{grgss}
Let $A$ be a finite dimensional, ring-indecomposable, non-semisimple $\Z$-graded $k$-algebra such that $A/J_A$ is separable over $k$. Then, the following are equivalent for $a \in \Z$ and $n>0$.
\begin{enumerate}
\item $\Om_A^n(S) \simeq S(a)$ in $\md^\Z\!A$ for any simple objects in $\md^\Z\!A$.
\item $A$ is self-injective and there exists a graded algebra automorphism $\a$ of $A$ such that $\Om^n \simeq (-)_\a(a)$ as functors on $\smd^\Z\!A$ and $P_\a \simeq P$ in $\md^\Z\!A$ for all $P \in \proj^\Z\!A$..
\item There exists a graded algebra automorphism $\a$ of $A$ such that $\Om^n_{A^e}(A) \simeq {}_1A_\a(a)$ in $\md^\Z\!A^e$ and $P_\a \simeq P$ in $\md^\Z\!A$ for all $P \in \proj^\Z\!A$.
\end{enumerate}
\end{Cor}

Similarly, we name these algebras as follows.
\begin{Def}
A finite dimensional graded algebra is {\it $(a)$-twisted $n$-periodic} if it is a direct product of simple algebras or algebras satisfying the equivalent conditions in Corollary \ref{grgss}.
\end{Def}

\section{Auslander correspondence}\label{8}
We now prove the first main result Theorem \ref{intf} of this paper, which gives a homological characterization of the Auslander algebras of finite triangulated categories.

First, we give the properties of the endomorphism algebra of a basic additive generator for a finite triangulated category, proving Theorem \ref{intf} (\ref{triang}) $\Rightarrow$ (\ref{tw3}).

\begin{Prop}\label{fineasy}
Let $\T$ be a $k$-linear, $\rHom$-finite idempotent-complete triangulated category. Assume $\T$ has an additive generator $M$. Take $M$ to be basic and set $C=\End_\T(M)$. Let $\s$ be the automorphism of $C$ induced by $[1]$; precisely, fix an isomorphism $a \colon M \to M[1]$ and define $\a$ by $\a(f)=a^{-1} \circ f[1] \circ a$ for $f \in \End_\T(M)$. Then, $C$ is a finite dimensional algebra which is twisted $3$-periodic.
\end{Prop}
\begin{proof}Since $\md\T \simeq \md C$ and $\md\T$ is a Frobenuis category (see \cite[4.2]{Kr}), $C$ is self-injective. Also, since the triangles in $\T$ yield projective resolutions of $C$-modules, the third syzygy is induced by the automorphism $\a$, that is, we have $\Om^3 \simeq (-)_\a$ on $\smd C$. Then $C$ is twisted $3$-periodic by Corollary \ref{twgss}. \end{proof}

For the converse implication, we need the following result due to Amiot, which allows one to introduce a triangle structure on the category of projectives in a Frobenius category.
\begin{Prop}\cite[8.1]{Am}\label{Ami}
Let $\P$ be an idempotent complete $k$-linear category such that the functor category $\md \P$ is naturally a Frobenius category. Let $S$ be an autoequivalence of $\P$ and extend this to $\md\P \to \md\P$.
Assume there exists an exact sequence of exact functors from $\md \P$ to $\md \P$
\[ \xymatrix{ 0 \ar[r] & 1 \ar[r] & X^0 \ar[r] & X^1 \ar[r] & X^2 \ar[r] & S \ar[r] & 0, } \]
where $X^i$ take values in $\P=\proj\P$. Then, $\P$ has a structure of a triangulated category with suspension $S$. The triangles are ones isomorphic to $X^0M \to X^1M \to X^2M \to SX^0M$ for $M \in \md\P$.
\end{Prop}

Combining this with Corollary \ref{twgss}, we can prove Theorem \ref{intf} (\ref{tw3}) $\Rightarrow$ (\ref{triang}).
Let us summarize the proof below.
\begin{proof}[Proof of Theorem \ref{intf}]
(\ref{triang}) $\Rightarrow$ (\ref{tw3}) is Proposition \ref{fineasy}.\\
(\ref{tw3}) $\Rightarrow$ (\ref{triang})
Since $A$ is self-injective, $\Om^3$ permutes the simples, so by Corollary {\ref{twgss}}, there exists an exact sequence 
\[ \xymatrix{ 0 \ar[r] & A \ar[r] & P^0 \ar[r] & P^1 \ar[r] & P^2 \ar[r] & {}_1A_{\a} \ar[r] & 0  } \]
of $(A,A)$-bimodules, with $P^i$'s projective and $\a$ is an automorphism of $A$. Then, we can apply Proposition {\ref{Ami}} for $\P=\proj A$, $S=-\otimes_A A_{\a}$, and $X^i=-\otimes_A P^i$.
\end{proof}

Applying a recent result of Keller \cite{Ke18}, we can formulate Theorem \ref{intf} in terms of bijection between triangulated categories and algebras under the standardness $\T$.
\begin{Thm}\label{std}
Let $k$ be an algebraically closed field. Then, there exists a bijection between the following.
\begin{enumerate}
\item\label{fas} The set of triangle equivalence classes of $k$-linear, $\rHom$-finite, idempotent-complete triangulated categories which are finite, algebraic, and standard.
\item\label{mes} The set of isomorphism classes of finite dimensional mesh algebras over $k$.
\end{enumerate}
The correspondence from {\rm (\ref{fas})} to {\rm (\ref{mes})} is given by taking the basic Auslander algebra, and from (\ref{mes}) to (\ref{fas}) by taking the category of projective modules.
\end{Thm}
\begin{proof}
We first check that each map is well-defined.

Let $\T$ be a triangulated category as in (\ref{fas}). Then, the standardness of $\T$ implies that its basic Auslander algebra is a mesh algebra.

Suppose next that $A$ is a finite dimensional mesh algebra. We want to show that $\proj A$ has the unique structure of an algebraic triangulated category up to equivalence. Since the third syzygy of simple $A$-modules are simple, $\T=\proj A$ has a structure of a triangulated category by Theorem \ref{intf}. Also, this is standard since $A$ is a mesh algebra. We claim that $\proj A$ admits a triangle structure which is algebraic. Since $\T$ is a finite, standard triangulated category, there exists a Dynkin quiver $Q$, a $k$-linear automorphism $F$ of $\der(\md kQ)$, and a $k$-linear equivalence $\der(\md kQ)/F \simeq \proj A$ \cite{Rie}. As in the proof of \cite{Ke18}, $F$ is isomorphic to $-\otimes_{kQ}^L X$ for some $(kQ,kQ)$-bimodule complex $X$. Then by \cite{Ke05}, $\der(\md kQ)/F$ admits an algebraic triangle structure as a triangulated orbit category, hence so does $\proj A$. This finishes the proof of the claim. Now, this algebraic triangle structure is unique up to equivalence by the main result of \cite{Ke18}. This shows the well-definedness.

It is clear that these maps are mutually inverse.
\end{proof}

\section{Graded projectivization}\label{easy}
In this section, we formulate the method of realizing certain additive categories, which we call {\it{$G$-finite}} additive categories on which a group $G$ acts with some finiteness conditions, as the category of graded projective modules over a $G$-graded algebra. This generalizes the classical `projectivization' \cite[\uppercase\expandafter{\romannumeral2}.2]{ARS}, which realizes a finite additive category as the category of projectives over an algebra.

Let $\A$ be an additive category with an action of a group $G$. Precisely, an automorphism $F_g$ of $\A$ is given for each $g \in G$ so that $F_{gh}=F_h \circ F_g$ for all $g, h \in G$. Then the action of $G$ extends to an automorphism of $\md\A$ by $F_gM=M \circ F_g^{-1}$. For example, the action on the representable functors is $F_g\A(-,X)=\A(-,F_gX)$.

Recall that the orbit category $\A/G$ has the same objects as $\A$ and the morphism space 
\[ (\A/G)(X,Y)=\bigoplus_{g \in G}\A(X,F_gY) \]
and the composition $b \circ a$ of $a \in \T(X,F_gY)$ and $b \in \T(Y,F_hZ)$ is given by $b \circ a= F_g(b)a$, where the right hand side is the composition in $\A$. Then, $\A/G$ is naturally a $G$-graded category whose degree $g$ part is $\A(X,F_gY)$.
\begin{Prop}\label{gracat}
Let $\A$ be an additive category with an action of a group $G$. Consider the orbit category $\C=\A/G$. Then, the following assertions hold:
\begin{enumerate}
\item\label{yon} The Yoneda embedding $\A \to \proj^G\!\C$ is fully faithful. It is an equivalence if $\A$ is idempotent-complete.
\item\label{fun} There exists an equivalence $\md\A \simeq \md^G\!\C$ such that the action of $F_g$ on $\md\A$ corresponds to the grade shift $(g)$ on $\md^G\!\C$, that is, we have the following commutative diagram of functors:
\[ \xymatrix{
   \md\A \ar[r]^\simeq\ar[d]_{F_g} & \md^G\!\C \ar[d]^{(g)} \\
   \md\A \ar[r]^\simeq & \md^G\!\C. } \]
\end{enumerate}
\end{Prop}
\begin{proof}
(\ref{yon})  We have the Yoneda lemma for graded functors: $\Hom_{\md^G\!\C}(\C(-,X),M)=(MX)_0$. It follows that the Yoneda embedding $\A \to \proj^G\!\C$ is fully faithful. Also, if $\A$ is idempotent-complete, the projectives in $\md^G\!\C$ are representable, and therefore the Yoneda embedding is dense.\\
(\ref{fun})  It is clear that the functor in (\ref{yon}) induces an equivalence $\md\A \simeq \md^G\!\C$. Also, the degree $h$ part of the functor $\C(-,F_gX)$ is $\A(-,F_hF_gX)=\A(-,F_{gh}X)$, which is equal to the same degree part of $\C(-,X)(g)$. Thus we have the commutative diagram.
\end{proof}

Now we impose the following finiteness conditions on the $G$-action:
\begin{enumerate}
\def\theenumi{(G{\arabic{enumi}})}
\def\labelenumi{\theenumi}
\item\label{g1} There is $M \in \A$ such that $\A=\add\{F_gM \mid g \in G \}$.
\item\label{g2} For any $X,Y \in \A$, $\Hom_\A(X,F_gY)=0$ for almost all $g \in G$.
\end{enumerate}
If these conditions are satisfied, we say that an additive category $\A$ with an action of $G$ is {\it{$G$-finite}}. If $\A$ is a $G$-finite additive category, we say $M \in \A$ as in \ref{g1} is a {\it{$G$-additive generator}}. If $G$ is generated by a single element $F$, we use the term {\it{$F$-finite}} for $G$-finiteness, and {\it{$F$-additive generator}} for $G$-additive generator.
Note that if $G$ is the trivial group, $G$-finiteness is nothing but finiteness, and a $G$-additive generator is an additive generator.

Let us reformulate Proposition \ref{gracat} in terms of the graded endomorphism algebra below. Note that this generalizes the classical `projectivization' for finite additive categories, which is the case $G$ is trivial, to `graded projectivization' for $G$-finite categories. Although this is rather formal, it will be useful in the sequel.

\begin{Prop}\label{efu}
Let $\A$ be a $k$-linear, $\rHom$-finite, idempotent-complete category with an action of $G$, which is $G$-finite. Let $M \in \A$ be a $G$-additive generator and set $C=\End_{\A/G}(M)$. Then, the following assertions hold:
\begin{enumerate}
\item $C$ is a finite dimensional $G$-graded algebra.
\item The functor $\A \to \proj^G\!C$, \ $X \mapsto \bigoplus_{g \in G}\Hom_\A(M,F_gX)$ is an equivalence.
\item\label{sansansan} There exists an equivalence $\md\A \simeq \md^G\!C$ such that the action of $g$ on $\md\A$ corresponds to the grade shift $(g)$ on $\md^G\!C$.
\end{enumerate}
\end{Prop}
\begin{proof}
(1)  $C$ is finite dimensional by \ref{g2}. \\
(2)  Since we have an equivalence $\proj^G\!\A/G \to \proj^G\!C$ by substituting $M$, the assertion follows from Proposition \ref{gracat} (\ref{yon}).\\
(3)  This is the same as Proposition \ref{gracat} (\ref{fun}).
\end{proof}

\begin{Def}\label{grmo}
$G$-Graded rings $A$ and $B$ are {\it{graded Morita equivalent}} if there is an equivalence $\md^G\!A \simeq \md^G\!B$ which commutes with grade shift functors $(g)$ for all $g \in G$.
\end{Def}

Let us note the following remark.
\begin{Prop}\label{grm}
Assume {\rm \ref{g1}} is satisfied and set $C=\End_{\A/G}(M)$.
\begin{enumerate}
\item\label{grm1} The ungraded algebra $C$ does not depend on the choice of $M$ up to Morita equivalence.
\item\label{grm2} The graded algebra $C$ does not depend on the choice of $M$ up to graded Morita equivalence.
\end{enumerate}
\end{Prop}
\begin{proof}
(1)  Since $C$ is the endomorphism algebra of an additive generator of the category $\A/G$, the assertion follows.\\
(2)  This follows from Proposition \ref{efu} (\ref{sansansan}).
\end{proof}


As a direct application of this graded projectivization, we present as an example the following graded version of the Auslander correspondence. For simplicity, we consider $\Z$-graded algebras. A graded algebra $\Lam$ is {\it{representation-finite}} if $\md^\Z\!\Lam$ has finitely many indecomposables up to grade shift. This is equivalent to the representation-finiteness of the ungraded algebra $\Lam$ \cite{GG}.
\begin{Prop}
There exists a bijection between the following.
\begin{enumerate}
\item\label{lam} The set of graded Morita equivalence classes of finite dimensional $\Z$-graded algebras $\Lam$ of finite representation type.
\item\label{gam} The set of graded Morita equivalence classes of finite dimensional $\Z$-graded algebras $\Gam$ with $\gd \Gam \leq 2 \leq \dd \Gam$.
\end{enumerate}
The correspondence is given as follows:
\begin{itemize}
\item From {\rm (\ref{lam})} to {\rm (\ref{gam})}: $\Gam=\End_\Lam(M)=\bigoplus_{n \in \Z}\Hom_\Lam(M,M(n))_0$ for a $(1)$-additive generator $M$ for $\md^\Z\!\Lam$.
\item From {\rm (\ref{gam})} to {\rm (\ref{lam})}:  $\Lam=\End_\Gam(Q)=\bigoplus_{n \in \Z}\Hom_\Gam(Q,Q(n))_0$ for a $(1)$-additive generator $Q$ for the category of graded projective-injective $\Gam$-modules.
\end{itemize}
\end{Prop}
\begin{proof}Note that $\Gam$ (resp. $\Lam$) does not depend on the choice of $M$ (resp. $Q$) by Proposition \ref{grm} (\ref{grm2}). The rest of the proof follows by the same argument as in Theorem \ref{Au}; see \cite[VI.5]{ARS}. \end{proof}

Notice that this correspondence $\Lam \leftrightarrow \Gam$ is the same as the ungraded case, thus it is a refinement of Theorem \ref{Au} on how much grading $\Lam$ or $\Gam$ have up to graded Morita equivalence.

\section{Uniqueness of triangle structures}\label{unithm}
The aim of this section is to prove some results which state the uniqueness of triangle structures on certain additive categories.
We say that an additive category {\it $\C$ has a unique algebraic triangle structure up to equivalence} if $\C_1=(\C,[1],\triangle)$ and $\C_2=(\C,[1]^\prime,\triangle^\prime)$ are algebraic triangle structures on $\C$, then there exists a triangle equivalence $F \colon \C_1 \xrightarrow{\simeq} \C_2$ such that $F(X) \simeq X$ in $\C$ for all $X \in \C$.

The following is the main result of this section.
\begin{Thm}\label{unique}
Let $\Lam$ be a ring with no simple ring summands such that $\K(\proj\Lam)$ is Krull-Schmidt. Then, the additive category $\K (\proj\Lam)$ has a unique algebraic triangle structure up to equivalence.
\end{Thm}

We give applications of Theorem \ref{unique}.
For a quiver $Q$, let $\Z Q$ be the associated infinite translation quiver \cite{ASS,Hap}, and let $k(\Z Q)$ be its mesh category \cite{Hap}.
\begin{Cor}\label{mesh}
Let $Q$ be a disjoint union of Dynkin quivers which does not contain $A_1$. Then, the mesh category $k(\Z Q)$ has a unique algebraic triangle structure up to equivalence.
\end{Cor}

As a consequence, we have the classification of $[1]$-finite algebraic triangulated categories.
\begin{Thm}\label{cor}
Let $k$ be an algebraically closed field. Any $[1]$-finite algebraic triangulated category over $k$ is triangle equivalent to the bounded derived category $\der(\md kQ)$ of the path algebra $kQ$ for a disjoint union $Q$ of Dynkin quivers of type A, D, and E.
\end{Thm}

Now we start the preparations for the proofs of the above results.
Recall that an additive category is {\it{Krull-Schmidt}} if any object is a finite direct sum of objects whose endomorphism rings are local. This is the case if the category is idempotent-complete and $\rHom$-finite over a complete Noetherian local ring. A Krull-Schmidt category $\C$ is {\it{purely non-semisimple}} if for each $X \in \C$, $J_\C(-,X) \neq 0$ or $J_\C(X,-)\neq 0$ holds. Note that these conditions are equivalent if $\C$ is triangulated.

First we observe that the suspension and the terms appearing in triangles in a triangulated category are determined by its additive structure under some Krull-Schmidt assumptions. 

\begin{Lem}\label{rmin}
Let $\C$ be a Krull-Schmidt additive category. Assume $\C$ has a structure of a triangulated category. Let $f \colon X \to Y$ be a right minimal morphism in $J_\C$.
\begin{enumerate}
\item\label{wkcok} The mapping cone of $f$ is the minimal weak cokernel of $f$.
\item\label{wkcokwkcok} $X[1]$ is the minimal weak cokernel of the minimal weak cokernel of $f$.
\end{enumerate}
\end{Lem}
\begin{proof}
Complete $f$ to a triangle $X \xrightarrow{f} Y \xrightarrow{g} Z \xrightarrow{h} X[1]$. \\
(\ref{wkcok})  We have to show that $g$ is the minimal weak cokernel of $f$. We only have to show the (left) minimality. If this is not the case, then $h$ has a summand $W \xrightarrow{1_W} W$ for a common non-zero summand $W$ of $Z$ and $X[1]$. This contradicts the right minimality of $f$.\\
(\ref{wkcokwkcok})  We want to show that $h$ is the minimal weak cokernel of $g$. Again, we only have to show the minimality. If this is not the case, then $f[1]$ has a summand $V \xrightarrow{1_V} V$ for a common non-zero summand $V$ of $X[1]$ and $Y[1]$. This contradicts $f \in J_\C(X,Y)$.
\end{proof}

We deduce that the possible triangle structures on a given purely non-semisimple Krull-Schmidt additive category is roughly unique in the following sense.
We denote by $\mathrm{cone}_\triangle(f)$ the mapping cone of $f$ in a triangle structure $\triangle$.
\begin{Prop}\label{aaaa}
Let $\C$ be a purely non-semisimple Krull-Schmidt additive category. If $(\C,[1],\triangle)$ and $(\C,[1]^\prime,\triangle^\prime)$ are triangle structures on $\C$, then we have the following.
\begin{enumerate}
\item\label{wise} $X[1] \simeq X[1]^\prime$ for all objects $X \in \C$.
\item\label{corn} $\mathrm{cone}_{\triangle}(f) \simeq \mathrm{cone}_{\triangle^\prime}(f)$ in $\C$ for all morphisms $f$ in $\C$.
\end{enumerate}
\end{Prop}
\begin{proof}
(1)  Let $X \in \C$ be an indecomposable object. Since $\C$ is purely non-semisimple, there exists a non-zero morphism $f \colon X \to Y$ in $J_\C$. Then, $f$ is a right minimal radical map, and hence the assertion follows from Lemma \ref{rmin} (\ref{wkcokwkcok}).\\
(2)  Let $f \colon X \to Y$ be an arbitrary morphism in $\C$. By removing the summands isomorphic to $W \xrightarrow{1} W$, which does not affect the mapping cone, we may assume $f \in J_\C$.  Then, $f$ has a decomposition $X_1 \oplus X_2 \xrightarrow{(f_1,0)} Y$ with right minimal $f_1 \in J_\C$ and the mapping cone of $f$ is the direct sum of that of $f_1$ and $X_2[1]$. Now the mapping cone of $f_1$ is determined by Lemma \ref{rmin} (\ref{wkcok}) and since $\C$ is purely non-semisimple, $[1]$ is determined by the additive structure by (1). This proves the assertion.
\end{proof}

For Theorem \ref{unique}, we need the following result of Keller on algebraic triangulated categories.
\begin{Prop}\cite[4.3]{Ke}\label{Kel}
Let $\T$ be an algebraic triangulated category and $T \in \T$ be a tilting object. Then, there exists a triangle equivalence $\T \simeq \K (\proj \End_{\T}(T))$. 
\end{Prop}

Note that we have the following observation, which will be crucial for the proof.
\begin{Lem}\label{tilt}
Let $\C$ be a purely non-semisimple Krull-Schmidt additive category. Assume $\C_1=(\C,[1],\triangle)$ and $\C_2=(\C,[1]^\prime,\triangle^\prime)$ are triangle structures on $\C$. Then, an object $T \in \C$ is a tilting object in $\C_1$ if and only if it is a tilting object in $\C_2$.
\end{Lem}
\begin{proof}
Indeed, we have $\C_1(T,T[n])=\C_2(T,T[n]^\prime)$ by Proposition \ref{aaaa} (\ref{wise}), which shows that the vanishing of extensions does not depend on the triangle structure. Also, by Proposition \ref{aaaa} (\ref{corn}), $T$ generates $\C_1$ if and only if $T$ generates $\C_2$. This shows the assertion.
\end{proof}

Now we are ready to prove our results.
\begin{proof}[Proof of Theorem \ref{unique}] Assume $\C$ is triangulated. We show that $\C$ is triangle equivalent to $\cK=\K (\proj\Lam)$ by finding a tilting object whose endomorphism algebra  is $\Lam$. Fix an additive equivalence $\C \simeq \cK$. Then, $\cK$ is purely non-semisimple and Krull-Schmidt by our assumption on $\Lam$. Let $T \in \C$ be the object corresponding to $\Lam \in \cK$. Then, $T$ is a tilting object by Lemma \ref{tilt} and clearly $\End_\C(T)=\Lam$. By our assumption that $\C$ is algebraic, we deduce that $\C$ is triangle equivalent to $\cK$ by Proposition {\ref{Kel}}.
\end{proof}

For the proof of Corollary \ref{mesh}, let us recall the following standardness theorem of Riedtmann.
\begin{Prop}\cite{Rie}\label{Ri}
Let $k$ be a field and $\T$ be a $k$-linear, $\rHom$-finite idempotent-complete triangulated category whose AR-quiver is $\Z Q$ for some acyclic quiver $Q$. Assume the endomorphism algebra of an indecomposable object of $\T$ is $k$. Then, $\T$ is $k$-linearly equivalent to the mesh category $k(\Z Q)$.
\end{Prop}
A well known application of this result is an equivalence $\K(\proj kQ) \simeq k(\Z Q)$ for a Dynkin quiver $Q$ \cite[I. 5.6]{Hap}.
\begin{proof}[Proof of Corollary \ref{mesh}] Since $k(\Z Q) \simeq \K(\proj kQ)$ as additive categories, Theorem \ref{unique} gives the result. \end{proof}

A $k$-linear triangulated category $\T$ is {\it{locally finite}} \cite{XZ} if for each indecomposable $X \in \T$, we have $\sum_{Y : \text{indec.}}\dim_k\Hom_\T(X,Y) < \infty$. This condition is equivalent to its dual \cite{XZ}. Clearly, our $[1]$-finite triangulated categories are locally finite.
The classification of $[1]$-finite triangulated category depends on the following result.
\begin{Prop}\cite[2.3.5]{XZ}\label{AR}
Let $k$ be an algebraically closed field and $\T$ be a locally finite triangulated category which does not contain a non-zero finite triangulated subcategory. Then, the AR-quiver of $\T$ is $\Z Q$ for a disjoint union $Q$ of Dynkin quivers of type A, D, and E.
\end{Prop}
\begin{proof}[Proof of Theorem \ref{cor}]
The AR-quiver of a $[1]$-finite triangulated category is $\Z Q$ for some Dynkin quiver $Q$ by Proposition \ref{AR}. Moreover, it is equivalent to $k(\Z Q)$ by Proposition \ref{Ri}. Thus Corollary \ref{mesh} applies. \end{proof}

We end this section by noting the following lemma, which we use later.
This lemma states in particular, that for mesh categories, the suspension is unique up to isomorphism of functors
\begin{Lem}\label{susp}
Let $Q$ be a Dynkin quiver and $\a$ be an automorphism of the mesh category $k(\Z Q)$ such that $\a X \simeq X$ for all $X \in k(\Z Q)$. Then, $\a$ is isomorphic as functors to the identity functor.
\end{Lem}
\begin{proof} Since $Q$ is Dynkin, we can inductively construct a natural isomorphism between $\a$ and $\mathrm{id}$. \end{proof}

\section{$[1]$-Auslander correspondence}\label{main}
In this section, we prove the second main result Theorem \ref{int1} of this paper. In the first subsection, we give the correspondence from triangulated categories to algebras, and the converse one in the second subsection. We will prove the main theorem in the final subsection.
\subsection{From triangulated categories to algebras}
We apply the graded projectivization prepared in Section \ref{easy} to triangulated categories. Let $\T$ be a $k$-linear, $\rHom$-finite, idempotent-complete triangulated category. Consider the action on $\T$ of $G=\Z$, generated by the suspension $[1]$. Then, the $G$-finiteness in this case are
\begin{enumerate}
\def\theenumi{(S{\arabic{enumi}})}
\def\labelenumi{\theenumi}
\item\label{s1} There is $M \in \T$ such that $\T=\add\{M[n] \mid n \in \Z \}$.
\item\label{s2} For any $X,Y \in \T$, $\Hom_\T(X,Y[n])=0$ for almost all $n$.
\end{enumerate}
According to the terminology in Section \ref{easy}, we say $\T$ is {\it{$[1]$-finite}}, and call $M$ as in \ref{s1} a {\it{$[1]$-additive generator}}.

The following theorem gives the correspondence from triangulated categories to algebras.
\begin{Prop}\label{easy one}
Let $\T$ be a $k$-linear, $\rHom$-finite, idempotent-complete, triangulated category which is $[1]$-finite. Let $M \in \T$ be a $[1]$-additive generator and set $C=\End_{\T/[1]}(M)$. Then, $C$ is a finite-dimensional graded self-injective algebra such that $\Om^3L \simeq L(-1)$ for any graded $C$-module $L$.
\end{Prop}
\begin{proof}$C$ is finite dimensional by \ref{s2}. Also, since $\md\T \simeq \md^\Z\!C$ by Proposition \ref{efu} (\ref{fun}) and $\md\T$ is Frobenius, $C$ is self-injective. It remains to show the statement on the third syzygy. Let $L$ be a graded $C$-module and let $Q \to R \to L \to 0$ be a projective presentation of $L$ in $\md^\Z\!C$. Take the map $X \to Y$ in $\T$ corresponding to $Q \to R$ and complete it to a trianle $W \to X \to Y \to W[1]$.  Put $P_Z=\bigoplus_{n \in \Z}\Hom_\T(M,Z[n])$ for each $Z \in \T$. This is the graded projective $C$-module corresponding to $Z$. Note that $P_{Z[1]}=P_Z(1)$, where $(1)$ is the grade shift functor on $\md^\Z\!C$. The triangle above yields an exact sequence $P_X(-1) \to P_Y(-1) \to P_W \to P_X \to P_Y \to P_W(1)$. Since $P_X=Q$ and $P_Y=R$, we see that $\Om^3L \simeq L(-1)$.
\end{proof}

\begin{Ex}
Let $Q$ be a Dynkin quiver and $\T =\der(\md kQ)$. Let $M$ be an additive generator for $\md kQ$. Then, $M$ is a $[1]$-additive generator for $\T$ and we have $C=\End_{kQ}(M) \oplus \Ext_{kQ}^1(M,M)$. The degree 0 part of $C$ is the Auslander algebra of $\md kQ$.

Let $Q'$ be another Dynkin quiver with the same underlying graph $\Delta$ as $Q$. Since $kQ$ and $kQ'$ are derived equivalent, we have $\der (\md kQ')=\T$. Similarly as above, an additive generator $M'$ for $\md kQ'$ is a $[1]$-additive generator for $\T$. The corresponding graded algebra $C'$ is $\End_{kQ'}(M') \oplus \Ext_{kQ'}^1(M',M')$, with the Auslander algebra of $\md kQ'$ in the degree 0 part.

By Lemma \ref{grm}, $C$ and $C'$ are isomorphic as ungraded algebras (but not as graded algebras).
In this way, $C \simeq C'$ contains the Auslander algebras of module categories over $\Delta$ for any orientation of $\Delta$.
\end{Ex}

Let us give a more specific example.
\begin{Ex}\label{a3}
Let $Q$ be the following Dynkin quiver of type $A_3$, and $\T$ be its derived category $\der(\md kQ)$.
\[ \xymatrix@C=5mm{a & b \ar[l] & c \ar[l]. } \]
Then, the AR-quiver of $\T$ is as follows:
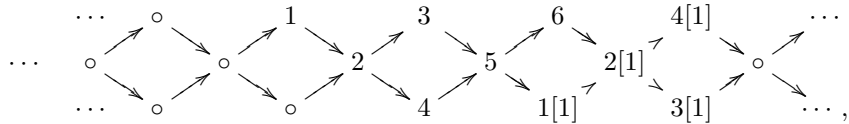
\[ {
   \xymatrix@!C@!R@C=1.2pt@R=1.2pt{ & \cdots & \circ\ar[dr] & & 1 {} \ar[dr] & & \ 3 {} \ \ar[dr] & & 6 {} \ar[dr] & & 4[1]\ar[dr] & & \cdots \\
                            \cdots & \circ\ar[dr]\ar[ur] & & \circ \ar[dr]\ar[ur] & & 2 {} \ar[dr]\ar[ur] & & 5 {} \ar[dr]\ar[ur] & & 2[1]\ar[dr]\ar[ur] & & \circ\ar[dr]\ar[ur] \\
& \cdots & \circ\ar[ur] & & \circ \ar[ur] & & 4 {} \ar[ur] & & 1[1]\ar[ur] & & 3[1]\ar[ur] & & \cdots, } 
}\]
where $1, \ldots, 6$ denotes the objects from $\md kQ$. Take $M=\bigoplus_{i=1}^6 M_i$, where $M_i$ is the indecomposable $kQ$-module corresponding to the vertex $i$. Then, $C=\End_{kQ}(M) \oplus \Ext_{kQ}^1(M,M)$. It is easily verified that $C$ is presented by the quiver $\Z A_3/[1]$ and the mesh relations. The quiver of $C$ looks as follows:
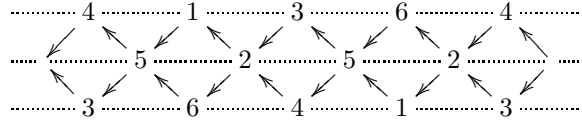
\[ \xymatrix@C=3mm@R=2mm{ & & 4 \ar[dl]\ar@{.}[ll] & & 1 \ar[dl]\ar@{.}[ll] & & 3 \ar[dl]\ar@{.}[ll] & & 6 \ar[dl]\ar@{.}[ll] & & 4 \ar[dl]\ar@{.}[ll] && \ar@{.}[ll]\\
                         & \ar@{.}[l] & & 5 \ar@{.}[ll]\ar[ul]\ar[dl] & & 2 \ar@{.}[ll]\ar[ul]\ar[dl] & & 5 \ar[ul]\ar[dl]\ar@{.}[ll] & & 2\ar@{.}[ll]\ar[ul]\ar[dl] & & \ar@{.}[ll]\ar[ul]\ar[dl] & \ar@{.}[l] \\
                          & & 3 \ar[ul]\ar@{.}[ll] & & 6 \ar[ul]\ar@{.}[ll] & & 4 \ar[ul]\ar@{.}[ll] & & 1 \ar[ul]\ar@{.}[ll] & & 3 \ar[ul]\ar@{.}[ll] &&\ar@{.}[ll] .} \]
where the vertices with the same number are identified, with mesh relations along the dotted lines. The arrows $1 \to 5$ and $2 \to 6$ have degree 1 and all the others have degree 0.

Now, let $Q'$ be the quiver obtained by reflecting $Q$ at vertex $a$;
\[ \xymatrix@C=5mm{a \ar[r] & b & c \ar[l]. } \]
Fix an equivalence $\der(\md kQ') \simeq \der(\md kQ)$ so that $M'=M_2 \oplus \cdots \oplus M_6 \oplus M_1[1]$ is an additive generator for $\md kQ'$. Then, $C'=\End_{kQ'}(M') \oplus \Ext_{kQ'}^1(M',M')$ is presented by the same quiver with relations as $C$, with arrows $2 \to 1$ and $2 \to 6$ having degree 1 and all the others degree 0. Thus $C \simeq C'$ as ungraded algebras but not as graded algebras.

Nevertheless, $C$ and $C'$ are graded Morita equivalent. Here we give a direct equivalence $\md^\Z\!C \to \md^\Z\!C'$. Let $e_i$ be the idempotent of $C$ corresponding to $M_i \ (1 \leq i \leq 6)$ and set $P=e_2C \oplus \cdots \oplus e_6C \oplus e_1C(1)$. Then, we have $\End_C(P) \simeq C'$ as graded algebras and $\Hom_C(P,-)$ gives a desired equivalence.
\end{Ex}

\subsection{From algebras to triangulated categories}
We can give the converse correspondence as in Section \ref{8}. Setting $a=-1$ in the following theorem gives the result.
\begin{Prop}\label{not easy one}
Let $A$ be a finite dimensional graded algebra such that $A/J_A$ is separable over $k$ and $\Om^3S \simeq S(a)$ for any graded simple module $S$. Then, $\proj^\Z\!A$ has a structure of a triangulated category. If $k$ is algebraically closed and $a \neq 0$, then the suspension is isomorphic to $(-a)$ and the algebraic triangle structure on $\proj^\Z\!A$ is unique up to equivalence.
\end{Prop}
\begin{proof}By Corollary {\ref{grgss}}, $A$ is self-injective and there exists an exact sequence
\[ \xymatrix{ 0 \ar[r] & A \ar[r] & P^0 \ar[r] & P^1 \ar[r] & P^2 \ar[r] & {}_1A_{\a}(-a) \ar[r] & 0 } \]
in $\md^\Z\!A^e$, where $P^i, i=0,1,2$ are projectives, and $\a$ is a graded algebra automorphism of $A$ such that $P_\a \simeq P$ for all $P \in \proj^\Z\!A$. Then, we can apply Proposition {\ref{Ami}} for $\P=\proj^\Z\!A$, $X^i=-\otimes_A P^i$ and $S=(-)_\a(-a)$ to see that $\proj^\Z\!A$ is triangulated with suspension $(-)_\a(-a)$. 
Now assume $k$ is algebraically closed and $a \neq 0$. Since we have $\Hom_{\proj^\Z\!A}(X,Y(-na))=0$ for almost all $n \in \Z$ for each $X, Y \in \proj^\Z\!A$, the triangulated category $\proj^\Z\!A$ is $[1]$-finite, and therefore, it is equivalent to the mesh category $k(\Z Q)$ for some Dynkin diagram $Q$ by Propositions \ref{AR} and \ref{Ri}. Then, by changing the triangle structure if necessary, $\proj^\Z\!A$ has a structure of an algebraic triangulated category, which is unique up to equivalence by Corollary {\ref{mesh}}. Also, $(-)_\a(-a)$ and $(-a)$ are isomorphic as functors by Lemma \ref{susp}.
\end{proof}

\subsection{Proof of Theorem \ref{int1}}
Combining the previous results, we can now prove the second main result of this paper.
\begin{proof}[Proof of Theorem \ref{int1}]For $M$ as in (\ref{T}), $C$ is as stated in (\ref{A}) by Proposition {\ref{easy one}}. Also, the graded Morita equivalence class of $C$ does not depend on the choice of $M$ by Proposition \ref{grm}. 
This shows the well-definedness of (\ref{T}) to (\ref{A}).

For the map from ({\ref{A}}) to ({\ref{T}}), it is well-defined since $\proj^\Z\!C$ has the unique structure of an algebraic triangulated category up to equivalence by Proposition {\ref{not easy one}}.

It is easily checked that these maps are mutually inverse.

The bijection between (\ref{T}) and (\ref{Dyn}) is Proposition \ref{AR} and Corollary {\ref{cor}}.
\end{proof}

\begin{Rem}
The algebra $C$ in Theorem \ref{int1} satisfies $[3] \simeq (1)$ as functors on $\smd^\Z\!C$ by Proposition \ref{easy one}.
\end{Rem}

\section{Applications to Cohen-Macaulay modules}\label{app}
Applying our classification in Theorem \ref{cor} of $[1]$-finite triangulated categories, we show that the stable categories $\sCM^\Z\!\Lam$ of some CM-finite Iwanaga-Gorenstein algebras, in particular, of (commutative) graded simple singularities are triangle equivalent to the derived categories of Dynkin quivers.

A Noetherian algebra $\Lam$ is {\it{Iwanaga-Gorenstein}} if $\id_\Lam \Lam=\id_{\Lam^\op}\Lam < \infty$. A typical example of Iwanaga-Gorenstein algebra is given by commutative Gorenstein rings of finite Krull dimension.
For an Iwanaga-Gorenstein algebra $\Lam$, we have the category
\[ \CM\Lam=\{ X \in \md\Lam \mid \Ext_\Lam^i(X,\Lam)=0 \ \text{for all} \ i >0 \} \]
of {\it{Cohen-Macaulay}} $\Lam$-modules. It is naturally a Frobenius category and we have a triangulated category $\sCM\Lam$.

Now consider the case $\Lam$ is graded: let $\Lam=\bigoplus_{n \geq 0}\Lam_n$ is a positively graded Noetherian algebra such that each $\Lam_n$ is finite dimensional over a field $k$. 
If $\Lam$ is a graded Iwanaga-Gorenstein algebra, we similarly have the category
\[ \CM^\Z\!\Lam =\{ X \in \md^\Z\!\Lam \mid \Ext_\Lam^i(X,\Lam)=0 \ \text{for all} \ i >0 \} \]
of graded Cohen-Macaulay modules. It is again Frobenius and hence the stable category $\sCM^\Z\!\Lam$ is triangulated.
A graded Iwanaga-Gorenstein algebra is {\it{CM-finite}} if $\CM^\Z\Lam$ has finitely many indecomposable objects up to grade shift. 

We now show that CM-finite Iwanaga-Gorenstein algebras give a large class of examples of $[1]$-finite triangulated categories.
\begin{Prop}\label{sing}
Let $\Lam$ be a positively graded CM-finite Iwanaga-Gorenstein algebra with $\gd\Lam_0 < \infty$. Then, the triangulated category $\sCM^\Z\!\Lam$ is $[1]$-finite.
\end{Prop}

To prove this, we need an observation for general Noetherian algebras, which is motivated by \cite[3.5]{Ya}.
Let us fix some notations. We denote by $\Ext_\Lam^i(-,-)_0$ the Ext groups on $\md^\Z\!\Lam$. Note that for $M, N \in \md^\Z\!\Lam$, the Ext groups on $\md \Lam$ are graded $k$-vector spaces: $\Ext_\Lam^i(M,N)=\bigoplus_{n \in \Z}\Ext_\Lam^i(M,N(n))_0$, $(i \geq 0)$. For each $M \in \md^\Z\!\Lam$ and $n \in \Z$, we denote by $M_{\geq n}$ the $\Lam$-submodule of $M$ consisting of components of degree $\geq n$.

\begin{Lem}\label{nex}
Let $\Lam$ be a positively graded Noetherian algebra with $\gd\Lam_0 < \infty$. Then, for any $X,Y \in \md^\Z\!\Lam$, we have $\Hom_\Lam(X,\Om^nY)_0=0$ for sufficiently large $n$.
\end{Lem}
\begin{proof}
Take a minimal graded projective resolution of $Y$: $\cdots \to P_2 \to P_1 \to P_0 \to Y \to 0$. We will show that for each $i \in \Z$, $P_n=(P_n)_{\geq i}$ holds for $n\gg0$. For this, it suffices to show that $P_n=(P_n)_{\geq 1}$ for $Y=Y_{\geq 0}$. Note that the degree 0 part of the minimal projective resolution of $Y$ yields a $\Lam_0$-projective resolution of $Y_0$. By our assumption that $\gd\Lam_0 < \infty$, we have $(P_n)_0=0$, hence $(P_n)_{\geq 1}=P_n$ for sufficiently large $n$. 
Now,  we have $\Hom_\Lam(X,\Lam(-n))_0=\Hom_\Lam(X,\Lam)_{-n}=0$ for $n \gg 0$. Indeed, this is certainly true if $X$ is projective. For general $X$, take a surjection $P \twoheadrightarrow X$ from a projective module $P$. Then we have an injection $\Hom_\Lam(X,\Lam) \hookrightarrow \Hom_\Lam(P,\Lam)$ and our assertion follows from the case $X$ is projective.
Therefore, we conclude that $\Hom_\Lam(X,P_n)_0=0$, thus $\Hom_\Lam(X,\Om^{n+1}Y)_0=0$ for sufficienly large $n$.
\end{proof}

\begin{proof}[Proof of Proposition \ref{sing}]
We verify the conditions \ref{shift} and \ref{simconn}. 

First we show \ref{simconn}: $\sHom_\Lam(X,\Om^nY)_0= 0$ for almost all $n$ for each $X,Y \in \underline{\CM}^\Z(\Lam)$.
The case $n \gg 0$ is done in Lemma \ref{nex}, so it remains to prove the case $n \ll 0$. Since $\Lam$ is CM-finite, $\sCM^\Z\!\Lam$ has the AR duality, and we have $D\sHom(X,\Om^nY)_0 \simeq \sHom(Y,\Om^{-n-1}\tau X)_0$, hence the assertion follows from the case of $n \gg 0$.

Next we show \ref{shift}: $\sCM^\Z\!\Lam$ has only finitely many indecomposables up to suspension.
Since $\Lam$ is of finite CM type, there exists $0 \neq n \in \Z$ such that $\Om^nX \simeq X$ up to grade shift for any indecomposable $X \in \sCM^\Z\!\Lam$. By \ref{simconn}, $\Om^nX$ and $X$ are not actually isomorphic in $\sCM^\Z\!\Lam$. Therefore, $\sCM^\Z\!\Lam$ has only fintely many indecomposables up to $\Om^n$, in particular up to $\Om^{-1}$.

These assertions show that $\sCM^\Z\!\Lam$ is $[1]$-finite.
\end{proof}

As an application of Theorem \ref{cor}, we immediately obtain the following result.
\begin{Thm}\label{CM}
Let $k$ be algebraically closed and let $\Lam=\bigoplus_{n \geq 0}\Lam_n$ is a positively graded Iwanaga-Gorenstein algebra such that each $\Lam_n$ is finite dimensional over $k$. Suppose $\Lam$ is CM-finite and $\gd\Lam_0<\infty$. Then, the AR-quiver of $\sCM^\Z\!\Lam$ is $\Z\Delta$ for a disjoint union $\Delta$ of some Dynkin diagrams of type A, D and E. Moreover, $\sCM^\Z\!\Lam$ is triangle equivalent to $\der(\md kQ)$ for any orientation $Q$ of $\Delta$.
\end{Thm}
\begin{proof} The statement for the AR-quiver follows from Proposition \ref{sing} and Proposition \ref{AR}. The triangle equivalence follows from Proposition {\ref{sing}} and Theorem {\ref{cor}}. \end{proof}

A well-known class of commutative Gorenstein rings of finite representation type is given by simple singularities. Here we assume that $k$ is algebraically closed of characteristic 0. Then, they are classified up to isomorphism by the Dynkin diagrams for each $d=\dim\Lam$ and have the form $k[x,y,z_2,\ldots,z_d]/(f)$ with
\begin{itemize}
\item[($A_n$)] $f=x^2+y^{n+1}+z_2^2+\cdots+z_d^2$, \quad $(n \geq 1)$,
\item[($D_n$)] $f=x^2y+y^{n-1}+z_2^2+\cdots+z_d^2$, \quad $(n \geq 4)$,
\item[($E_6$)] $f=x^3+y^4+z_2^2+\cdots+z_d^2$,
\item[($E_7$)] $f=x^3+xy^3+z_2^2+\cdots+z_d^2$,
\item[($E_8$)] $f=x^3+y^5+z_2^2+\cdots+z_d^2$,
\end{itemize}
see \cite[Chapter 9]{LW}.
We admit any gradings on $\Lam$ so that each variable and $f$ are homogeneous of positive degrees. Then, $\Lam$ is CM-finite (in the graded sense) since its completion $\widehat{\Lam}$ at the maximal ideal $\Lam_{>0}$ is CM-finite, that is, $\CM\widehat{\Lam}$ has only finitely many indecomposable objects \cite[Chapter 15]{Y}. 

\begin{Cor}
Let $k$ be an algebraically closed field of characteristic zero and $\Lam=k[x,y,z_2,\ldots,z_d]/(f)$ with $f$ one of above. Give a grading on $\Lam$ so that each variable and $f$ are homogeneous of positive degrees. Then, the stable category $\sCM^\Z\!\Lam$ is triangle equivalent to the derived category $\der(\md kQ)$ of the path algebra $kQ$ of a disjoint union $Q$ of Dynkin quivers.
\end{Cor}

We give several more examples. First we consider the case $\Lam$ is finite dimensional.
\begin{Ex}
Let
\[ \Lam=\Lam_n=k[x]/(x^n) \]
with $\deg x=1$. Then, $\Lam$ is a finite dimensional self-injective algebra. In this case we have $\CM^\Z\!\Lam=\md^\Z\!\Lam$. It is of finite representation type with indecomposable $\Lam$-modules $\Lam_i$ ($1 \leq i \leq n$), and $\Lam_0=k$ has finite global dimension. We can easily compute its AR-quiver (for $n=4$) to be
\[ \xymatrix@!C=3pt@!R=3pt{ & \cdots & \circ\ar[dr] & & \Lam_3(-1) {} \ar[dr] & & \ \Lam_3 {} \ \ar[dr] & & \Lam_3(1) {} \ar[dr] & & \circ\ar[dr] & & \cdots \\
                            \cdots & \circ\ar[dr]\ar[ur] & & \Lam_2(-2) \ar[dr]\ar[ur] & & \Lam_2(-1) {} \ar[dr]\ar[ur] & & \Lam_2 {} \ar[dr]\ar[ur] & & \Lam_2(1)\ar[dr]\ar[ur] & & \circ\ar[dr]\ar[ur] \\
                                      & \cdots & \circ\ar[ur] & & \Lam_1(-2)\ar[ur] & & \Lam_1(-1) {} \ar[ur] & & \Lam_1\ar[ur] & & \Lam_1(1)\ar[ur] & & \cdots, } \]
where the top of $\Lam_i$ is in degree 0. We see that the AR-quiver of $\smd^\Z\!\Lam$ is $\Z A_{n-1}$. Consequently, we have a triangle equivalence $\smd^\Z\!\Lam \simeq \der(\md kQ)$ for a quiver $Q$ of type $A_{n-1}$.
\end{Ex}

The next one is a finite dimensional Iwanaga-Gorenstein algebra.
\begin{Ex}
Let $\Lam$ be the algebra presented by the following quiver with relations:
\[ \xymatrix@C=12mm@R=5mm{ 1 \ar[ddr]_x & 2 \ar[l]_a\ar[ddr]_y & 3 \ar[l]_b \\ \\
                  4 \ar[uu]^c & 5 \ar[l]^f\ar[uu]^d & 6 \ar[l]^g\ar[uu]^e }
\qquad
\xymatrix@R=3mm{
da=fc, \ eb=gd, \\
ax=yg, \\
cx=0, \ xd=0, \ xf=0, \ dy=0, \ by=0, \ ye=0, }
\]
with $\deg x=\deg y=1$ and all other arrows having degree $0$. Then, it is an Iwanaga-Gorenstein algebra of dimension $1$. (In fact, this is the 3-preprojective algebra \cite{IO} of its degree $0$ part.) We can compute the AR-quiver of $\md^\Z\!\Lam$ to be the following.
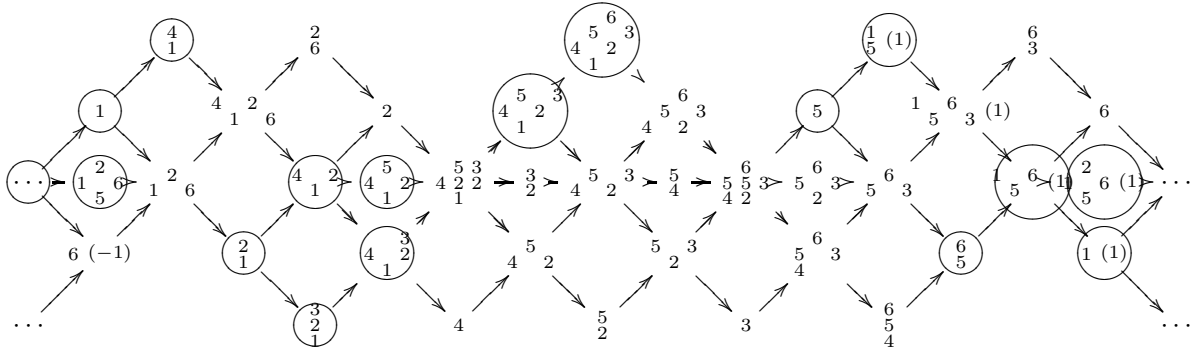
\newcommand{\simp}[1]{\begin{smallmatrix}#1\end{smallmatrix}}
\newcommand{\nagni}[2]{\begin{smallmatrix}#1 \\ #2 \end{smallmatrix}}
\newcommand{\san}[3]{\begin{smallmatrix}#1 \\ #2 \\ #3 \end{smallmatrix}}
\newcommand{\niichi}[3]{\begin{smallmatrix}#1 & & #2 \\ & #3 & \end{smallmatrix}}
\newcommand{\ichini}[3]{\begin{smallmatrix}& #1 & \\ #2 & & #3 \end{smallmatrix}}
\newcommand{\ichiniichi}[4]{\begin{smallmatrix} & #1 & \\ #2 && #3 \\ & #4 &\end{smallmatrix}}
\newcommand{\nini}[4]{\begin{smallmatrix}& #1 && #2 \\ #3 && #4 & \end{smallmatrix}}
\newcommand{\mimi}[4]{\begin{smallmatrix}#1 & & #2 & \\ & #3 & & #4 \end{smallmatrix}}
\newcommand{\ichisan}[4]{\begin{smallmatrix}& & #1 \\ #2 & & #3 \\ & #4 & \end{smallmatrix}}
\newcommand{\sanichi}[4]{\begin{smallmatrix}& #1 & \\ #2 & & #3 \\ #4 & & \end{smallmatrix}}
\newcommand{\radi}{\begin{smallmatrix}& 5 & & 3 \\ 4 & & 2 & \\ & 1 & &\end{smallmatrix}}
\newcommand{\socle}{\begin{smallmatrix} & & 6 & \\ & 5 & & 3 \\ 4 & & 2 & \end{smallmatrix}}
\newcommand{\saidai}{\begin{smallmatrix} & & 6 & \\ & 5 & & 3 \\ 4 & & 2 & \\ & 1 & & \end{smallmatrix}}
\newcommand{\rokutop}{\begin{smallmatrix} & 6 & \\ 5 & 5 & 3 \\ 4 & 2 & \end{smallmatrix}}
\newcommand{\rokusoc}{\begin{smallmatrix} & 5 & 3 \\ 4 & 2 & 2 \\ & 1 & \end{smallmatrix}}
\newcommand{\gr}[1]{\begin{smallmatrix} (#1) \end{smallmatrix}}
\[ \xymatrix@!C=1mm@!R=1mm{
    & &*+[Fo:<3mm>]{\nagni{4}{1}}\ar[dr] && {\nagni{2}{6}}\ar[dr] &&&& *+[Fo:<5mm>]{\saidai}\ar[dr] &&&& *+[Fo:<4mm>]{\nagni{1}{5}\gr{1}}\ar[dr] && {\nagni{6}{3}}\ar[dr] && \\
    & *+[Fo:<3mm>]{\simp{1}}\ar[dr]\ar[ur] && {\mimi{4}{2}{1}{6}}\ar[ur]\ar[dr] && {\simp{2}}\ar[dr] && *+[Fo:<5mm>]{\radi}\ar[ur]\ar[dr] && {\socle}\ar[dr] && *+[Fo:<3mm>]{\simp{5}}\ar[ur]\ar[dr] && {\mimi{1}{6}{5}{3}\gr{1}}\ar[dr]\ar[ur] && {\simp{6}}\ar[dr] & \\
    *+[Fo:<3mm>]{\cdots}\ar[r]\ar[ur]\ar[dr] & *+[Fo:<4mm>]{\ichiniichi{2}{1}{6}{5}}\ar[r] & {\ichini{2}{1}{6}}\ar[ur]\ar[dr] && *+[Fo:<4mm>]{\niichi{4}{2}{1}}\ar[r]\ar[ur]\ar[dr] & *+[Fo:<4mm>]{\ichiniichi{5}{4}{2}{1}}\ar[r] & {\rokusoc}\ar[r]\ar[ur]\ar[dr] & {\nagni{3}{2}}\ar[r] & {\nini{5}{3}{4}{2}}\ar[r]\ar[ur]\ar[dr] & {\nagni{5}{4}}\ar[r] & {\rokutop}\ar[r]\ar[ur]\ar[dr] & {\ichiniichi{6}{5}{3}{2}}\ar[r] & {\ichini{6}{5}{3}}\ar[ur]\ar[dr] && *+[Fo:<5mm>]{\niichi{1}{6}{5}\gr{1}}\qquad\ar[r]\ar[ur]\ar[dr] & *+[Fo:<5mm>]{\ichiniichi{2}{1}{6}{5}\gr{1}}\ar[r] & \cdots \\
& {\simp{6}\gr{-1}}\ar[ur] && *+[Fo:<3mm>]{\nagni{2}{1}}\ar[ur]\ar[dr] && *+[Fo:<4mm>]{\ichisan{3}{4}{2}{1}}\ar[ur]\ar[dr] && {\ichini{5}{4}{2}}\ar[ur]\ar[dr] && {\niichi{5}{3}{2}}\ar[ur]\ar[dr] && {\sanichi{6}{5}{3}{4}}\ar[ur]\ar[dr] && *+[Fo:<3mm>]{\nagni{6}{5}}\ar[ur] && *+[Fo:<4mm>]{\simp{1}\gr{1}}\ar[ur]\ar[dr] & \\
\cdots\ar[ur] &&&& *+[Fo:<3mm>]{\san{3}{2}{1}}\ar[ur] && {\simp{4}}\ar[ur] && {\nagni{5}{2}}\ar[ur] && {\simp{3}}\ar[ur] && {\san{6}{5}{4}}\ar[ur] &&&& \cdots 
    } \]
Here, each module is graded so that its top is concentrated in degree $0$, or equivalently, its lowest degree is at $0$.
We then compute the category $\CM^\Z\!\Lam$ to be the circled modules and it is verified that the AR-quiver of $\sCM^\Z\!\Lam$ is
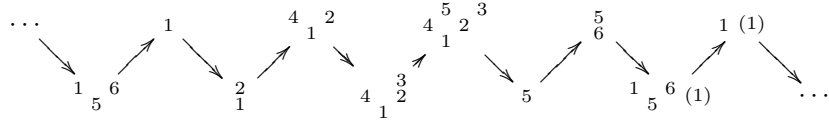
\[ \xymatrix@!C=1mm@!R=1mm{
   \cdots \ar[dr] && {\simp{1}}\ar[dr] && {\niichi{4}{2}{1}}\ar[dr] && {\radi}\ar[dr] && {\nagni{5}{6}}\ar[dr] && {\simp{1}\gr{1}}\ar[dr] \\
   & {\niichi{1}{6}{5}}\ar[ur] && {\nagni{2}{1}}\ar[ur] && {\ichisan{3}{4}{2}{1}}\ar[ur] && {\simp{5}}\ar[ur] && {\niichi{1}{6}{5}\gr{1}}\ar[ur] && \cdots } \]
We see that this is $\Z A_2$ and consequently $\sCM^\Z\!\Lam \simeq \der(\md kQ)$ for a quiver $Q$ of type $A_2$.
\end{Ex}

We consider as a final example a {\it Gorenstein order}: let $R=k[x_1,\ldots,x_d]$ be a polynomial ring. A Noetherian $R$-algebra $\Lam$ is an {\it{$R$-order}} if it is projective as an $R$-module. An $R$-order $\Lam$ is {\it{Gorenstein}} if $\Hom_R(\Lam,R)$ is projective as a $\Lam$-module.
In this case, Cohen-Macaulay $\Lam$-modules are $\Lam$-modules which are projective as $R$-modules.
\begin{Ex}
Let $R=k[x]$ be a graded polynomial ring with $\deg x=1$ and let
\[ \Lam=\begin{pmatrix}  R & R \\ (x^n) & R \end{pmatrix}. \]
This is a Gorenstein $R$-order of dimension 1. Its indecomposable CM modules up to grade shift are given by the row vectors $M_i=\begin{pmatrix} (x^i) & R \end{pmatrix}$ for $0 \leq i \leq n$, and $M_0$ and $M_n$ are the projectives. We define the gradings on $M_i$'s so that their top $\begin{pmatrix} 0 & k \end{pmatrix}$ is in degree 0. Then, the AR-quiver of $\CM^\Z\!\Lam$ (for $n=4$) is computed to be
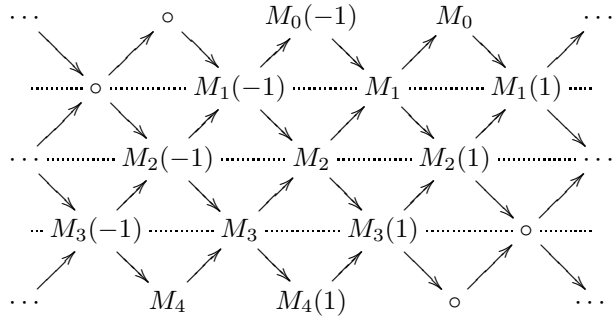
\[ \xymatrix@!C=1mm@!R=1mm{
   \cdots \ar[dr] && \circ \ar[dr] && M_0(-1)\ar[dr] && M_0\ar[dr] && \cdots \\
   & \circ \ar@{.}[l]\ar[dr]\ar[ur] && M_1(-1)\ar@{.}[ll]\ar[dr]\ar[ur] && M_1\ar@{.}[ll]\ar[dr]\ar[ur] && M_1(1)\ar@{.}[ll]\ar[dr]\ar[ur] & \ar@{.}[l]\\
   \cdots \ar[dr]\ar[ur] && M_2(-1)\ar@{.}[ll]\ar[dr]\ar[ur] && M_2\ar@{.}[ll]\ar[dr]\ar[ur] && M_2(1)\ar@{.}[ll]\ar[dr]\ar[ur] && \cdots \ar@{.}[ll] \\
   & M_3(-1)\ar@{.}[l]\ar[dr]\ar[ur] && M_3\ar@{.}[ll]\ar[dr]\ar[ur] && M_3(1)\ar@{.}[ll]\ar[dr]\ar[ur] && \circ \ar@{.}[ll]\ar[dr]\ar[ur] & \ar@{.}[l] \\
   \cdots \ar[ur] && M_4\ar[ur] && M_4(1)\ar[ur] && \circ \ar[ur] &&\cdots , } \]
where the upgoing arrows are natural inclusions, the downgoing arrows are the multiplications by $x$, and the dotted lines indicate the AR-translations.
By deleting the projective vertices, we see that the AR-quiver of $\sCM^\Z\!\Lam$ is $\Z A_{n-1}$, and consequently $\sCM^\Z\!\Lam \simeq \der(\md kQ)$ for a quiver $Q$ of type $A_{n-1}$.
\end{Ex}


\thebibliography{99}
\bibitem[Am]{Am} C. Amiot, {\it{On the structure of triangulated categories with finitely many indecomposables}}, Bull. Soc. math. France 135 (3), 2007, 435-474.
\bibitem[AHK]{AHK} L. Angeleri-H\"{u}gel, D. Happel, and H. Krause, {\it Handbook of tilting theory}, London Mathematical Society Lecture Note Series 332, Cambridge University Press, Cambridge, 2007.
\bibitem[AZ]{AZ} M. Artin and J. J. Zhang, {\it Noncommutative projective schemes}, Adv. Math. 109 (1994) 228-287.
\bibitem[Au]{A} M. Auslander, {\it{Representation dimension of Artin algebras}}, Queen Mary College mathematics notes, London, 1971.
\bibitem[AR]{AR} M. Auslander and I. Reiten, {\it{$D\mathrm{Tr}$-periodic modules and functors}}, CMS Conference Proceedings 18, American Mathematical Society, Province, RI, (1996) 39-50. 
\bibitem[ARS]{ARS} M. Auslander, I. Reiten and S. O. Smal\o, {\it{Representation theory of Artin algebras}}, Cambridge studies in advanced mathematics 36, Cambridge University Press, Cambridge, 1995.
\bibitem[As]{As} H. Asashiba, {\it{A generalization of Gabriel's Galois covering functors \uppercase\expandafter{\romannumeral2}: 2-categorical Cohen-Montgomery duality}}, Appl. Categor. Struct. (2017) 25, 155-186.
\bibitem[ASS]{ASS} I. Assem, D. Simson and A. Skowro\'nski, {\it{Elements of the representation theory of associative algebras, vol.1}}, London Mathematical Society Student Texts 65, Cambridge University Press, Cambridge, 2006.
\bibitem[BIY]{BIY} R. O. Buchweitz, O. Iyama, and K. Yamaura, {\it{Tilting theory for Gorenstein rings in dimension one}}, arXiv:1803.05269.
\bibitem[CYZ]{CYZ} X. W. Chen, Y. Ye, and P. Zhang, {\it Algebras of derived dimension zero}, Communications in Algebra, 36, 1-10 (2008).
\bibitem[CR1]{CR1} C.W. Curtis and I. Reiner, {\it{Methods of representation theory with applications to finite groups and orders, vol. 1}}, Pure and Applied Mathematics, Wiley-Interscience Publication, John Wiley \& Sons, New York, 1981.
\bibitem[CR2]{CR2} C.W. Curtis and I. Reiner, {\it{Methods of representation theory with applications to finite groups and orders, vol. 2}}, Pure and Applied Mathematics, Wiley-Interscience Publication, John Wiley \& Sons, New York, 1987.
\bibitem[E]{E} H. Enomoto, {\it{Classification of exact structures and Cohen-Macaulay-finite algebras}}, arXiv:1705.02163v3.
\bibitem[ES]{ES} K. Erdmann and A. Skowro\'nski, {\it Periodic algebras}, in: {\it Trends in representation theory of algebras and related topics}, EMS series of congress reports, European Mathematical Society, Z\"{u}rich, 2008.
\bibitem[GG]{GG} R. Gordon and E. L. Green, {\it{Representation theory of graded Artin algebras}}, J. Algebra 76 (1982) 138-152.
\bibitem[GSS]{GSS} E. L. Green, N. Snashall, and \O. Solberg, {\it{The Hochschild cohomology ring of a selfinjective algebra of finite representation type}}, Proc. Amer. Math. Soc. 131 (11) (2003), 3387-3393.
\bibitem[Ha]{Hap} D. Happel, {\it{Triangulated categories in the representation theory of finite dimensional algebras}}, London Mathematical Society Lecture Note Series 119, Cambridge University Press, Cambridge, 1988.
\bibitem[He]{He} A. Heller, {\it{Stable homotopy categories}}, Bull. Amer. Math. Soc. 74 (1), 1968, 28-63.
\bibitem[I1]{Iy1} O. Iyama, {\it{The relationship between homological properties and representation theoretic realization of artin algebras}}, Trans. Amer. Math. Soc. 357 (2) (2005), 709-734.
\bibitem[I2]{Iy2} O. Iyama, {\it{Auslander correspondence}}, Adv. Math. 210 (2007) 51-82.
\bibitem[IO]{IO} O. Iyama and S. Oppermann, {\it{Stable categories of higher preprojective algebras}}, Adv. Math. 244 (2013), 23-68.
\bibitem[Ke1]{Ke} B. Keller, {\it{Deriving DG categories}}, Ann. scient. \'Ec. Norm. Sup. (4)27(1) (1994) 63-102.
\bibitem[Ke2]{Ke05} B. Keller, {\it On triangulated orbit categories}, Doc. Math. 10 (2005), 551-581.
\bibitem[Ke3]{Ke2} B. Keller, {\it{On differntial graded categories}}, Proceedings of the International Congress of Mathematicians, vol. 2, Eur. Math. Soc, 2006, 151-190.
\bibitem[Ke4]{Ke18} B. Keller, {\it A remark on a theorem by C. Amiot}, arXiv:1806.00635v1.
\bibitem[Kr]{Kr} H. Krause, {\it{Derived categories, resolutions, and Brown representability}}, Interactions between homotopy theory and algebra, Comtemp. Math. 436, 101-139, Amer. Math. Soc, Province. RI, 2007.
\bibitem[KST]{KST} H. Kajiura, K. Saito, and A. Takahashi, {\it Matrix factorizations and representations of quivers II: Type ADE case}, Adv. Math. 211 (2007) 327-362.
\bibitem[LW]{LW} G. J. Leuschke and R. Wiegand, {\it{Cohen-Macaulay representations}}, vol. 181 of Mathematical Surverys and Monographs, American Mathematical Society, Province, RI, (2012).
\bibitem[N]{Nee} A. Neeman, {\it{Triangulated Categories}}, Annals of Mathematics Studies, vol. 148, Princeton University Press, 2001.
\bibitem[Ri]{Rie} C. Riedtmann, {\it{Algebren, Darstellungsk\"ocher, Uberlagerugen und zur\"uck}}, Comment. Math. Helvetici 55 (1980) 199-224.
\bibitem[Ro]{Ro} R. Rouquier, {\it Dimensions of triangulated categories}, J. K-Theory, 1 (2008), 193-256.
\bibitem[S]{S} D. Simson, {\it Linear representations of partially ordered sets and vector space categories}, Algebra, Logic, and Applications, vol. 4, Gordon and Breach Science Publishers, 1992.
\bibitem[XZ]{XZ} J. Xiao and B. Zhu, {\it{Locally finite triangulated categories}}, J. Algebra 290 (2005) 473-490.
\bibitem[Ya]{Ya} K. Yamaura, {\it{Realizing stable categories as derived categories}}, Adv. Math. 248 (2013) 784-819.
\bibitem[Yo1]{Y} Y. Yoshino, {\it{Cohen-Macaulay modules over Cohen-Macaulay rings}}, London Mathematical Society Lecture Note Series 146, Cambridge University Press, Cambridge, 1990.
\bibitem[Yo2]{Y2} Y. Yoshino, {\it A functorial approach to modules of G-dimension zero}, Illinois. J. Math. 49 (2005) 345-367.
\end{document}